\documentclass[a4paper,11pt]{scrartcl}
\usepackage[utf8]{inputenc}
\usepackage{amsmath}
\usepackage{amssymb}
\usepackage{mathtools}
\usepackage[shortlabels]{enumitem}
\usepackage[amsmath,thmmarks,hyperref,amsthm]{ntheorem}
\usepackage{csquotes}
\usepackage[pdftex]{hyperref}

\newtheorem{theorem}{Theorem}
\newtheorem{definition}[theorem]{Definition}
\newtheorem{lemma}[theorem]{Lemma}

\newtheorem{corollary}[theorem]{Corollary}

\theoremstyle{definition}
\newtheorem{remark}[theorem]{Remark}

\newcommand{\qbinom}[3]{\genfrac{[}{]}{0pt}{}{#1}{#2}_{#3}}
\newcommand{\qnumb}[2]{[{#1}]_{#2}}
\newcommand{\PS}{{\cal Q}}
\newcommand{\F}{\mathbb{F}}
\newcommand{\Deltalambda}{\Delta_\lambda}

\DeclareMathOperator{\Aut}{Aut}
\DeclareMathOperator{\Der}{Der}
\DeclareMathOperator{\lcm}{lcm}
\DeclareMathOperator{\PG}{PG}
\DeclareMathOperator{\Res}{Res}
\DeclareMathOperator{\rank}{rk}
\DeclareMathOperator{\Sp}{Sp}
\DeclareMathOperator{\Spb}{\textbf{Sp}}
\DeclareMathOperator{\codim}{codim}

\newcommand{\QWb}{\Spb\mathbf{(2r, q)}}

\newcommand{\QQp}{\Omega^+(2r,q)}
\newcommand{\QQ}{\Omega(2r+1,q)}
\newcommand{\QQm}{\Omega^-(2r+2, q)}

\DeclareMathOperator{\PGSp}{P\Gamma Sp}
\DeclareMathOperator{\PGO}{P\Gamma O}
\DeclareMathOperator{\PGU}{P\Gamma U}

\title{Designs in finite classical polar spaces}
\author{Michael Kiermaier\thanks{Department of Mathematics, University of Bayreuth, 95447 Bayreuth, Germany\\email: \texttt{michael.kiermaier@uni-bayreuth.de}\\homepage:~\url{https://mathe2.uni-bayreuth.de/michaelk/}
}
    \and
    Kai-Uwe Schmidt\thanks{Department of Mathematics, Paderborn University, 33098 Paderborn, Germany}
    \and
    Alfred Wassermann\thanks{Department of Mathematics\\University of Bayreuth, 95447 Bayreuth, Germany\\email: \texttt{alfred.wassermann@uni-bayreuth.de}}%
}

\begin{document}
\maketitle

\begin{abstract}
    Combinatorial designs have been studied for nearly 200 years.  Fifty years
    ago, Cameron, Delsarte, and Ray-Chaudhury started investigating their
    $q$-analogs, also known as subspace designs or designs over finite fields.

    Designs can be defined analogously in finite classical polar spaces, too.
    The definition includes the $m$-regular systems from projective geometry as
    the special case where the blocks are generators of the polar space.  The
    first nontrivial such designs for $t > 1$ were found by De Bruyn and
    Vanhove in 2012, and some more designs appeared recently in the PhD thesis
    of Lansdown.

    In this article, we investigate the theory of classical and subspace
    designs for applicability to designs in polar spaces, explicitly allowing
    arbitrary block dimensions.  In this way, we obtain divisibility conditions
    on the parameters, derived and residual designs, intersection numbers and
    an analog of Fisher's inequality.  We classify the parameters of symmetric
    designs.  Furthermore, we conduct a computer search to construct designs of
    strength $t=2$, resulting in designs for more than $140$ previously unknown
    parameter sets in various classical polar spaces over $\F_2$ and $\F_3$.
   
    \par\vskip\baselineskip\noindent
    \paragraph{Keywords.} Finite classical polar spaces, design theory, m-regular systems, Galois geometries.
\end{abstract}

\section{Introduction}

The classification of non-degenerate sesquilinear and non-singular quadratic forms on vector spaces over a finite field
gives rise to the \emph{finite classical polar spaces}.
Fixing such a form on a finite ambient vector space $V$, the polar space $\mathcal{Q}$ is formed by all subspaces of $V$
which are totally isotropic or totally singular.
The subspaces of maximal dimension are called \emph{generators} and their dimension is called the \emph{rank} of $\mathcal{Q}$.
The type of the form together with the order $q$ of the base field will be called the \emph{type} $Q$ of $\mathcal{Q}$.

A set $\mathcal{S}$ of generators of $\mathcal{Q}$ such that each point of $\mathcal{Q}$ is incident with
exactly $\lambda$ elements of $\mathcal{S}$ is called a \emph{$\lambda$-spread}.
The study of these objects goes back to Segre (1965) \cite{Segre-1965}.
Segre's work already contains a generalization replacing the points by subspaces of arbitrary dimension $t$, known as
a \emph{$\lambda$-regular system with regard to $(t-1)$-spaces}.
This definition clearly is related to the notion of combinatorial designs~\cite{handbook-2006} and
subspace designs~\cite{qdesigns2017}, where the ambient subset (or subspace) lattice has been replaced
by the subspaces of the polar space $\mathcal{Q}$, and it is covered by Delsarte's designs in association schemes,
see Vanhove~\cite[pp. 39, 81, 193]{vanhove2011}.

This naturally leads us to the following full adaption of the notion of a design to polar spaces.
A \emph{$t$-$(r, k, \lambda)_Q$ design} is a set $D$ of $k$-dimensional subspaces of the ambient polar space $\mathcal{Q}$
of type $Q$ such that every $t$-dimensional subspace of $\mathcal{Q}$ is incident with exactly
$\lambda$ elements (called \emph{blocks}) of $D$.
A design with $\lambda = 1$ is called \emph{Steiner system}.
The $\lambda$-regular systems with regard to the $(t-1)$-spaces form the special case $k = r$.

Apart from the elementary cases of trivial designs and designs coming from Latin-Greek halvings, hitherto only very few concrete constructions of designs in polar spaces of strength $t\geq 2$ have been known.
De Bruyn and Vanhove (2012) constructed designs with parameters $2$-$(3,3,2)_{\Omega(3)}$ and $2$-$(3,3,2)_{\Omega^-(2)}$, and showed the non-existence of $2$-$(3,3,2)_{\Sp(3)}$ design~\cite{DeBruyn2021}.
In the sequel, Bamberg and Lansdown \cite{bamberg2018,lansdown2020} constructed $2$-$(3,3,(q+1)/2)_{\Omega(q)}$ designs for $q\in\{5,7,11\}$.
Non-existence results for Steiner systems in polar spaces have been shown in \cite{Schmidt-Weiss-2022}, see Theorem~\ref{thm:schmidtweiss}.
By the recent non-constructive result~\cite{Weiss-2025-DCC93[4]:971-981}, nontrivial $t$-$(r,k,\lambda)_Q$ designs exist for all values of $t$ and all types $Q$, provided that $k > \frac{21}{2} t$ and $r$ is large enough.

In this article, we investigate the theory of classical and subspace designs for their applicability to designs in polar spaces.
After collecting the required preliminaries in Section~\ref{sec:prelim}, Section~\ref{sec:design} states the definition of a design $D$ in a polar space and gives the first basic results, including the supplementary (Lemma~\ref{lem:supplementary}) and the reduced (Lemma~\ref{lem:reduced}) design of $D$.
The latter leads to the notion of admissible parameters.
In Section~\ref{sec:der_res}, we construct derived (Theorem~\ref{thm:derived}) and residual (Theorem~\ref{thm:residual}) designs of $D$.
Equations for the intersection numbers of $D$ are computed in Section~\ref{sec:intersection} (Theorems~\ref{thm:mendelsohn} and~\ref{thm:koehler}) and yield a uniqueness result for the Latin-Greek halvings (Lemma~\ref{lem:latin_greek_unique}) in hyperbolic polar spaces.
In Section~\ref{sec:fisher}, we show that the Gram matrix of the point-block incidence matrix of a design $D$ of strength $t \geq 2$ is related to the adjacency matrix of the collinearity graph of the ambient polar space.
This leads to an analog of Fisher's inequality (Theorem~\ref{thm:fisher}), which in turn yields the classification of all $t$-$(r,r,1)_{\Omega^+(q)}$ Steiner systems of strength $t \geq 2$ (Theorem~\ref{thm:steiner_omegaplus}) by elementary means.
In the computational part in Section~\ref{sec:compute}, we construct designs with $t=2$ for more than $140$ previously unknown parameter sets in various classical polar spaces over $\F_2$ and $\F_3$.
Among these are the first nontrivial $2$-designs where the block dimension $k$ is strictly smaller than the rank $r$ of the ambient polar space.
We conclude the article stating a few open problems and suggestions for further research in Section~\ref{sec:outlook}.

\section{Preliminaries}\label{sec:prelim}
Throughout this article, $q \geq 2$ will denote a prime power and $V$ a vector space over $\F_q$ of finite dimension $v$.
The lattice $\PG(V) \cong \PG(n-1,q)$ of all subspaces of $V$ is a finite projective space of algebraic dimension $n$ and geometric dimension $n-1$.
For subspaces of projective and polar spaces, the word \emph{dimension} will always denote the algebraic dimension, as the resulting formulas tend to
be simpler and closer to the classical theory of block designs.
The set of all $\F_q$-subspaces of $V$ of dimension $k$ will be denoted by $\qbinom{V}{k}{q}$.
Its cardinality is given by the Gaussian binomial coefficient
\[
    \#\qbinom{V}{k}{q}
    = \qbinom{v}{k}{q}
    = \begin{cases}
        \displaystyle\prod_{i=1}^v \frac{q^{v-i+1}-1}{q^i - 1} & \text{for }k\in\{0,\ldots,v\}\text{;} \\
        0                                                      & \text{otherwise.}
    \end{cases}
\]
We will make use of the abbreviation $\qnumb{v}{q} = \qbinom{v}{1}{q}$, which is known as the \emph{$q$-analog of the number $v$}.
The subspaces of algebraic dimension $1$ will be called \emph{points}, of dimension $2$ \emph{lines}, of dimension $3$ \emph{planes} and
dimension $v - 1$ \emph{hyperplanes}.

For a subspace $U$ of $V$ of dimension $u$, $V/U$ is the ambient vector space of a projective space over $\F_q$ of dimension $v - u$.
As a consequence, the number of subspaces of $V$ of dimension $k$ containing $U$ is $\qbinom{v-u}{k-u}{q}$.

For an introduction to finite polar spaces, we refer to \cite{deBeuleKleinMetsch2011,Ball-2015-FiniteGeometry}.
Table~\ref{tbl:ps} shows all classical polar spaces of rank $r \geq 1$ over the finite field $\F_q$ up to isomorphism.
It will be convenient to collect all finite classical polar spaces $\PS$ which only differ (possibly) in their rank $r$ (but not in the type of
the underlying form nor the number $q$) into the \emph{type} $Q$.
The table includes the symbol we will use for the type, as well as a parameter $\epsilon$ which allows a uniform treatment
of all types of polar spaces in counting formulas.
Note that the Hermitian polar spaces only exist for squares $q$.
For even $q$, there is an isomorphism $\Omega(2t+1) \cong \Sp(2r,q)$.
For the readers' convenience, the second last column lists an alternative set of symbols that are sometimes found in the literature
(for example \cite{Schmidt-Weiss-2022}).
Each polar space has a natural embedding into a projective geometry $\PG(V)$.
The column $n$ lists the (algebraic) dimension $\dim(V)$ of this geometry, and column $G$ contains
the collineation group, see e.g. \cite[Section 2.3.5]{BrouwerVanMaldeghem2022}.

\begin{table}[htb]
    \caption{Finite classical polar spaces of rank $r$}
    \label{tbl:ps}
    \centering\begin{tabular}{lllllclll}
        name       & symbol $\PS$        & type $Q$      & $n$    & $\epsilon$ & \multicolumn{2}{l}{alternative symbols} & $G$                                     \\ \hline
        symplectic & $\Sp(2r, q)$        & $\Sp(q)$      & $2r$   & $0$        & $C_r$                                   & $W_{2r-1}(q)$   & $\PGSp_{2r}(q)$       \\
        Hermitian  & $U(2r, q)$          & $U^+(q)$      & $2r$   & $-1/2$     & ${}^2A_{2r-1}$                          & $H_{2r-1}(q)$   & $\PGU_{2r}(\sqrt{q})$ \\
        Hermitian  & $U(2r+1, q)$        & $U^-(q)$      & $2r+1$ & $+1/2$     & ${}^2A_{2r}$                            & $H_{2r}(q)$     & $\PGU_{2r+1}(\sqrt{q})$ \\
        hyperbolic & $\Omega^+(2r,q)$    & $\Omega^+(q)$ & $2r$   & $-1$       & $D_r$                                   & $Q^+_{2r-1}(q)$ & $\PGO^+_{2r}(q)$      \\
        parabolic  & $\Omega(2r+1,q)$    & $\Omega(q)$   & $2r+1$ & $0$        & $B_r$                                   & $Q_{2r}(q)$     & $\PGO_{2r+1}(q)$      \\
        elliptic   & $\Omega^-(2r+2, q)$ & $\Omega^-(q)$ & $2r+2$ & $+1$       & ${}^2D_{r+1}$                           & $Q^-_{2r+1}(q)$ & $\PGO^-_{2r+2}(q)$    \\
    \end{tabular}
\end{table}

In the following, $\PS$ will always denote a finite classical polar space over $\F_q$ of rank $r \geq 1$,
$Q$ will denote its type and $\epsilon$ will be its parameter as listed in Table~\ref{tbl:ps}.
Corresponding to the notation for vector spaces, the set of all subspaces of $\PS$ of dimension $k$ will be denoted by $\qbinom{\PS}{k}{Q}$, and its cardinality
will be denoted by $\qbinom{r}{k}{Q} \coloneqq \#\qbinom{\PS}{k}{Q}$.
Two distinct points $P,P'\in\qbinom{\PS}{1}{Q}$ are called \emph{collinear} if the line $\langle P,P'\rangle$ is contained in $\mathcal{Q}$.
For a subspace $U$ of the polar space $\PS$, the symbol $U^\perp$ denotes the orthogonal subspace of $U$ in $V$ with respect to the underlying form.
Note that $\dim(U) + \dim(U^\perp) = n$, and for all subspaces $U$ of $\mathcal{Q}$ we have $U \subseteq U^\perp$.

A hyperplane $H$ of the ambient vector space $V$ is either \emph{degenerate} or \emph{non-degenerate} with respect to the underlying form.
The degenerate hyperplanes are exactly the ones of the form $H = P^\perp$ with $P\in\qbinom{\PS}{1}{Q}$.
Moreover, all subspaces of $\PS$ contained in $H$ pass through the point $P$, and after modding out $P$, the resulting set of spaces is a polar space of the same
type $Q$ and rank one less.
For non-degenerate hyperplanes $H$, the restriction of $\PS$ to $H$ is again a polar space.

We need some basic counting formulas in the subspace poset of finite classical polar spaces.
A useful observation is that for a given subspace $U$ of dimension $u$ in the poset of all subspaces of the finite polar space $\PS$, the interval below $U$ is
the finite projective space $\PG(U) \cong \PG(u-1,q)$, and the interval above $U$ is the finite polar space $U^\perp / U$, which is of the same type $Q$ and of
rank $r - u$.
Therefore, counting subspaces below $U$ just resembles counting in finite projective spaces as discussed above.
For counting subspaces above $U$, we prepare the following lemma.

\begin{lemma}
    \label{lem:count}
    \item
    \begin{enumerate}[(a)]
        \item\label{lem:count:subspace}
        The number of $k$-dimensional subspaces of $\PS$ is equal to
        \[
            \qbinom{r}{k}{Q} = \qbinom{r}{k}{q}\cdot \prod_{i=r-k+1}^r (q^{i+\epsilon} + 1)\text{.}
        \]
        \item\label{lem:count:intermediate}
        The number of $k$-dimensional subspaces of $\PS$ containing a fixed $u$-dimensional subspace is
        \[
            \qbinom{r-u}{k-u}{Q} = \qbinom{r-u}{k-u}{q} \cdot \prod_{i=r-k+1}^{r-u}(q^{i+\epsilon} + 1)\text{.}
        \]
    \end{enumerate}
\end{lemma}

\begin{proof}
    Part~\ref{lem:count:subspace} is found in \cite[Lemma 9.4.1]{brouwer1989distance}.
    For an $u$-dimensional subspace $U$, Part~\ref{lem:count:intermediate} follows from an application of Part~\ref{lem:count:subspace} to the quotient
    polar space $U^\perp / U$, which is of the same type $Q$ and of rank $u$ less compared to the original polar space $\PS$.
\end{proof}

Again, we define the abbreviation $\qnumb{r}{Q} = \qbinom{r}{1}{Q} = \qnumb{r}{q} \cdot (q^{r + \epsilon} + 1)$.
Table~\ref{tab:hyper} lists the polar spaces $\PS'$ arising from non-degenerate hyperplanes of the ambient vector space of $\PS$, together with the
rank $r'$ and the parameter $\epsilon'$ of $\PS'$.
Note that in all cases, $r - r' \in \{0,1\}$ and $\epsilon' = \epsilon + 2(r-r') - 1$.
We remark that the symplectic space $\PS = \Sp(2r,q)$ does not show up in the table, since its point set $\qbinom{\PS}{1}{Q}$ equals the full point set of $\PG(V)$
(as all vectors are isotropic) and therefore, all hyperplanes are degenerate.

\begin{table}[htb]
\caption{Polar spaces $\PS'$ arising from non-degenerate hyperplanes of $\PS$.}\label{tab:hyper}
\[
    \begin{array}{lllll}
        \PS              & \epsilon & \PS'                   & r'  & \epsilon' \\
        \hline
        U(2r+1,q)        & +1/2     & U(2r,q)                & r   & -1/2      \\
        U(2r,q)          & -1/2     & U(2(r-1)+1,q)          & r-1 & +1/2      \\
        \Omega^+(2r,q)   & -1       & \Omega(2(r-1)+1,q)     & r-1 & 0         \\
        \Omega(2r+1,q)   & 0        & \Omega^+(2r,q)         & r   & -1        \\
        \Omega(2r+1,q)   & -1/2     & \Omega^-(2(r-1) + 2,q) & r-1 & +1/2      \\
        \Omega^-(2r+2,q) & +1       & \Omega(2r+1,q)         & r   & 0
    \end{array}
\]
\end{table}

The \emph{collinearity graph} of a polar space $\PS$ has the points of $\PS$ as vertex set, with two of them being adjacent if they are collinear.
This graph is known to be strongly regular with the parameters $(v,k,\lambda,\mu)$ where \cite{Kantor1982}
\begin{align*}
    v       & = \qbinom{r}{1}{Q}\text{,}                   &
    k       & = q\cdot \qbinom{r-1}{1}{Q}\text{,}          &
    \lambda & = q - 1 + q^2\cdot\qbinom{r-2}{1}{Q}\text{,} &
    \mu     & = \qbinom{r-1}{1}{Q}\text{.}
\end{align*}
From these parameters, the eigenvalues can be computed as $\theta_0 = k$ with multiplicity $m_0=1$ and
the two zeroes $\theta_1$, $\theta_2$ of the quadratic equation $\theta^2 - (\lambda - \mu)\theta - (k - \mu) = 0$.
The remaining multiplicities $m_1$ and $m_2$ are then determined by the system of linear equations $\theta_1 m_1 + \theta_2 m_2 = -1$ and $m_1 + m_2 = v - 1$.
This leads \cite[Theorem~2.2.12]{BrouwerVanMaldeghem2022} to the eigenvalues
\begin{align*}
    \theta_0 & = k\text{,}                       &
    \theta_1 & = q^{r-1} - 1\text{,}             &
    \theta_2 & = -(q^{r+\epsilon-1} + 1)\text{,}
\end{align*}
with multiplicities
\begin{align*}
    m_0 & = 1\text{,}                                                                                            \\
    m_1 & = q^{\epsilon + 1}\cdot\frac{q^{r+\epsilon-1} + 1}{q^\epsilon + 1}\cdot\qbinom{r}{1}{q}\quad\text{and} \\
    m_2 & = q \cdot\frac{q^{r+\epsilon} + 1}{q^\epsilon + 1}\cdot\qbinom{r-1}{1}{q}\text{.}
\end{align*}

\section{Designs in polar spaces}\label{sec:design}
Since the early 1970s, the notion of combinatorial block designs has been
generalized to \emph{subspace designs}, which are designs in finite projective spaces and
can be understood as a $q$-analog of the classical situation.
Many results in classical block designs turned out to have a generalization for subspace designs.
We refer the reader to \cite{handbook-2006} for a comprehensive treatment of combinatorial designs
and \cite{qdesigns2017} for an introduction and an overview to subspace designs.

The theory of subspace designs will serve as a blueprint for our investigation of designs in finite classical polar spaces.
As we will see, parts of the basic theory can also be applied in this situation, while others apparently do not have a direct adaption.
Again, $\PS$ will denote a finite classical polar space over $\F_q$ of rank $r \geq 1$, $Q$ will denote its type and $\epsilon$ will is its parameter
as listed in Table~\ref{tbl:ps}.

\begin{definition}
    Let $t,k\in\{0,\ldots,r\}$ and $\lambda$ be a non-negative integer.
    A set $D$ of subspaces of $\PS$ of (algebraic) dimension $k$ is called a \emph{$t$-$(r,k,\lambda)_Q$ design}
    (or a \emph{$t$-$(r,k,\lambda)$ design in $\PS$})
    if every subspace $T$ of $\PS$ of (algebraic) dimension $t$ is contained in exactly $\lambda$ elements of $D$.

    The elements of $D$ are called \emph{blocks} and the number $t$ is called the \emph{strength} of the design $D$.
    In the case $\lambda = 1$, $D$ is called a \emph{Steiner system}.
\end{definition}
We remark that our definition doesn't allow repeated blocks, so in design theory terminology, all designs are considered to be \emph{simple}.
Steiner systems with $t = 1$ are called \emph{spreads} and are well-studied objects, compare \cite{deBeuleKleinMetsch2011}.

In finite geometry, designs with parameters $t$-$(r,r,\lambda)_Q$, i.\,e. the dimensions of the blocks equal
the rank of the polar space, were already studied by Segre \cite{Segre-1965} under the name
\emph{$\lambda$-regular systems with respect to $(t-1)$-spaces}.
$\lambda$-regular systems with respect to points are called \emph{$\lambda$-spreads}, see \cite{cossidente2021regular} and \cite[Sec.~2.2.7]{BrouwerVanMaldeghem2022}.

Of course, the \emph{empty design} $D = \emptyset \subseteq \qbinom{\mathcal{Q}}{k}{Q}$ is a $t$-$(r,k,0)_q$ design for all values $t\in\{0,\ldots,k\}$.

\medskip
The following two lemmas are a direct consequence of Lemma~\ref{lem:count}\ref{lem:count:intermediate}.

\begin{lemma}
    The \emph{complete design} consisting of all $k$-dimensional subspaces of $\PS$ is a $t$-$(r,k,\lambda_{\max})_{Q}$ design with
    \[
        \lambda_{\max} = \qbinom{r-t}{k-t}{Q} = \qbinom{r-t}{k-t}{q} \cdot \prod_{i=r-k+1}^{r-t}(q^{i+\epsilon} + 1)\text{.}
    \]
\end{lemma}

\begin{lemma}\label{lem:supplementary}
    Let $D$ be a $t$-$(r,k,\lambda)_{Q}$ design.
    The \emph{supplementary} design $D^\complement = \qbinom{\PS}{k}{Q}\setminus D$ is a design with the parameters $t$-$(r,k,\qbinom{r-t}{k-t}{Q} - \lambda)_{Q}$.
\end{lemma}

By the last lemma, only the range $\lambda\in\{1,\ldots, \lfloor\lambda_{\max} / 2\rfloor\}$ is relevant for the investigation of designs.
Designs $D$ with $\lambda = \lambda_{\max}/2$ are called \emph{halvings} as they split $\qbinom{\PS}{k}{Q}$ into
two \enquote{halves} $D$ and $D^\complement$ which both are designs with the same parameters.
In the case $r = k$ halvings are also called \emph{hemisystems} \cite{cossidente2021regular}.

The empty and the complete design are supplementary to each other.
They are called the \emph{trivial} designs.

\begin{lemma}
    \label{lem:reduced}
    Let $D$ be a $t$-$(r,k,\lambda)_{Q}$ design.
    Then for each $s\in\{0,\ldots,t\}$, $D$ is an $s$-$(r,k,\lambda_s)_{Q}$ design with
    \[
        \lambda_s
        \coloneqq \lambda\cdot\dfrac{\qbinom{r-s}{t-s}{Q}}{\qbinom{k-s}{t-s}{q}}
        = \lambda\cdot\dfrac{\qbinom{r-s}{t-s}{q}}{\qbinom{k-s}{t-s}{q}} \cdot\prod_{i=r-t+1}^{r-s} (q^{i + \epsilon} + 1) \text{.}
    \]
    In particular, the number of blocks of $D$ is given by $\#D = \lambda_0$.
\end{lemma}

\begin{proof}
    Let $s\in\{0,\ldots,t\}$, $S\in\qbinom{\PS}{t}{Q}$ and $\lambda_S$ the number of blocks of $D$ containing $S$.
    Double counting the set $X$ of all pairs $(T,B)\in\qbinom{\PS}{t}{Q}\times D$ with $S \subseteq T \subseteq B$ gives
    \[
        \#X = \qbinom{r-s}{t-s}{Q} \cdot \lambda
        \qquad\text{and}\qquad
        \#X = \lambda_S \cdot \qbinom{k-s}{t-s}{q}\text{.}
    \]
    Therefore
    \[
        \lambda_S = \qbinom{r-s}{t-s}{Q} / \qbinom{k-s}{t-s}{q} \cdot \lambda
    \]
    does not depend on the choice of $S$.
\end{proof}

For $s = t-1$, the resulting design of Lemma~\ref{lem:reduced} is called the \emph{reduced} design of $D$.
The number $\lambda_1$ is called the \emph{replication number} of $D$.
We remark that $\lambda_t = \lambda$.
For a $t$-$(r,k,\lambda)_{Q}$ design, the numbers $\lambda_0, \ldots,\lambda_t$ clearly must be integers.
If these \emph{integrality conditions} are satisfied, the parameters $t$-$(r,k,\lambda)_{Q}$ will be called \emph{admissible}.
If for admissible parameters a design exists, the parameters are called \emph{realizable}.

For fixed $Q$, $r$, $k$ and $t$, the set of the admissible parameters will have the form $t$-$(r,k,\lambda)_Q$ with
$\lambda\in\{0,\Deltalambda,2\Deltalambda,\ldots,\lambda_{\max}\}$, where
\[
    \Deltalambda = \lcm \left\{\qbinom{k-s}{t-s}{q} / \gcd\Big(\qbinom{r-s}{t-s}{Q}, \qbinom{k-s}{t-s}{q}\Big) \mid s\in\{0,\ldots,t\}\right\}\text{.}
\]
In the important case $k = r$ we always get $\Deltalambda = 1$, implying that the integrality conditions do not give any further restrictions on the parameters.

\section{Derived and residual designs}\label{sec:der_res}
Now we investigate the restriction of designs to hyperplanes $H$.
Because of \cite{Kiermaier-Laue-2015}, resulting constructions can be seen as a polar space analog of the derived (Theorem~\ref{thm:derived})
and the residual (Theorem~\ref{thm:residual}) design.

We start with the easier case of a degenerate hyperplane $H$.
Here $H = P^\perp$ with $P\in\qbinom{\PS}{1}{Q}$ and we want to mod out $P$ to get a polar space.
\begin{theorem}\label{thm:derived}
    Let $D$ be a $t$-$(r,k,\lambda)_{Q}$ design of strength $t \geq 1$ and let $P$ be a point in $\PS$.

    Then
    \[
	    \Der_P(D) = \{B / P \in D\mid P \subseteq B\}
    \]
    is a $(t-1)$-$(r-1,k-1,\lambda)_Q$ design in the ambient polar space $P^\perp / P$.
    It is called the \emph{derived design} of $D$ in $P$.
\end{theorem}

\begin{proof}
    By the theory of polar spaces, $\Der_P(D)$ is contained in $\PS' = P^\perp / P$, which is a polar space of rank $r-1$ and the same type $Q$ as $\PS$.
    Let $T'\in\qbinom{\PS'}{t-1}{Q}$.
    Then $T' = T/P$ with $T\in\qbinom{\PS}{t}{Q}$ and $P \subseteq T$.
    By the representation $B' = B / P$ with $P \subseteq B$, the elements $B' \in \Der_P(D)$ with $T' \subseteq B'$ are in bijective
    correspondence with the blocks $B\in D$ with $T\subseteq B$.
    As $D$ is a $t$-design, their number equals $\lambda$.
\end{proof}

It remains the case of non-degenerate hyperplanes $H$.

\begin{theorem}\label{thm:residual}
    Let $D$ be a $t$-$(r,k,\lambda)_Q$ design and let $H$ be a non-degenerate hyperplane of $V$.
    The restriction of $\PS$ to $H$ is a polar space $\PS'$ of some rank $r'\in\{r,r-1\}$ and some type $Q'$.

    Then $\Res_H(D) = \{B\in D \mid B \subseteq H\}$ is a $(t-1)$-$(r', k, \lambda')_{Q'}$-design with
    \[
        \lambda'
        = \begin{cases}
            \dfrac{\qnumb{r-t+1}{q}\, (q^{r-k+\epsilon} + 1)}{\qnumb{k-t+1}{q}}\;\lambda & \text{if }r' = r\text{;}   \\[5mm]
            \dfrac{(q^{r-t+1+\epsilon} + 1)\, \qnumb{r-k}{q}}{\qnumb{k-t+1}{q}}\;\lambda & \text{if }r' = r-1\text{.}
        \end{cases}
    \]
    It is called the \emph{residual design} of $D$ in $H$.
\end{theorem}

\begin{proof}
    Let $T'\in\qbinom{\PS'}{t-1}{Q'}$ and $\Lambda = \{B' \in \Res_H(D) \mid T' \subseteq B'\}$.
    We count the set
    \[
        X = \Big\{(T,B)\in\qbinom{\PS}{t}{Q}\!\! \times D \mid T \subseteq B \text{ and }T \cap H = T'\Big\}
    \]
    in two ways.

    By $T \cap H = T'$, the involved subspaces $T$ necessarily contain $T'$.
    By Lemma~\ref{lem:count}\ref{lem:count:intermediate}, the number of all $\bar{T}\in \qbinom{\PS}{t}{Q}$ with $T' \subseteq \bar{T}$ is $\qnumb{r-t+1}{Q}$.
    Again by Lemma~\ref{lem:count}\ref{lem:count:intermediate}, the number of these $\bar{T}$ contained in $H$ is $\qnumb{r'-t+1}{Q'}$.
    The possible subspaces $T$ are exactly the remaining $\alpha \coloneqq \qnumb{r-t+1}{Q} - \qnumb{r'-t+1}{Q'}$ subspaces $\bar{T}$, as by the dimension formula
    they intersect $H$ in dimension $t-1$, which means $\bar{T} \cap H = T'$.
    Each such $T$ is contained in exactly $\lambda$ blocks of $D$.
    Hence $\#X = \alpha\lambda$.

    On the other hand, there are $\lambda_{t-1}$ ways to select $B\in D$ with $T' \subseteq B$.
    By the dimension formula, either $B \subseteq H$ (equivalently, $B\in \Lambda$) or $\dim(B \cap H) = k-1$.
    Only in the latter case, $B$ can be part of a pair in $X$.
    Fixing such a block $B$, we want to determine the number $\beta$ of $T\in \qbinom{B}{t}{q}$ with $T \cap H = T'$.
    These $T$ are of the form $T' + P$ with $P\in\qbinom{B}{1}{q} \setminus \qbinom{B\cap H}{1}{q}$.
    As there are $\qnumb{k}{q} - \qnumb{k-1}{q} = q^{k-1}$ choices for $P$, and any $\qnumb{t}{q} - \qnumb{t-1}{q} = q^{t-1}$ of them generate the same
    space $T = T' + P$, we get $\beta = \frac{q^{k-1}}{q^{t-1}} = q^{k-t}$.
    Therefore, $\#X = (\lambda_{t-1} - \#\Lambda)q^{k-t}$.

    Equating these two expressions for $\#X$ and using Lemma~\ref{lem:reduced}, we find
    \begin{align*}
        \#\Lambda
         & = \lambda_{t-1} - \frac{\alpha}{q^{k-t}}\lambda                                                                                \\
         & = \frac{1}{q^{k-t}}\Big(q^{k-t}\,\frac{\qnumb{r-t+1}{Q}}{\qnumb{k-t+1}{q}} - \qnumb{r-t+1}{Q} + \qnumb{r'-t+1}{Q'}\Big)\lambda \\
         & = \frac{\mu}{\qnumb{k-t+1}{q}\,q^{k-t}}\,\lambda
    \end{align*}
    with
    \[
        \mu = \qnumb{k-t+1}{q}\,\qnumb{r'-t+1}{Q'} - \qnumb{k-t}{q}\,\qnumb{r-t+1}{Q}\text{.}
    \]
    Since $\#\Lambda$ does not depend on the choice of $T'\in\qbinom{\PS'}{t-1}{Q'}$, we get that $\Res_H(D)$ is a $(t-1)$-$(r',k,\lambda')_{Q'}$
    design with $\lambda' = \#\Lambda$.

    Let $\epsilon'$ be the parameter of $\PS'$.
    For $r' = r$, we have $\epsilon' = \epsilon-1$ and thus
    \begin{align*}
        \mu & = \frac{\qnumb{r-t+1}{q}}{q-1}\Big((q^{k-t+1}-1)\, (q^{r-t+\epsilon} + 1) - (q^{k-t} - 1)\, (q^{r-t+1+\epsilon} + 1)\Big) \\
            & = \qnumb{r-t+1}{q}\, \frac{q^{k-t+1} - q^{r-t+\epsilon} - q^{k-t} + q^{r-t+1+\epsilon}}{q-1}                              \\
            & = \qnumb{r-t+1}{q}\, q^{k-t}\, (q^{r-k+\epsilon} + 1)\text{.}
    \end{align*}

    For $r' = r - 1$, we have $\epsilon' = \epsilon+1$ and therefore
    \begin{align*}
        \mu & = \frac{q^{r-t+1+\epsilon} + 1}{(q-1)^2}\Big((q^{k-t+1}-1)(q^{r-t}-1) - (q^{k-t} - 1)(q^{r-t+1} - 1)\Big) \\
            & = (q^{r-t+1+\epsilon} + 1)\, \frac{-q^{k-t+1} - q^{r-t} + q^{k-t} + q^{r-t+1}}{(q-1)^2}                   \\
            & = (q^{r-t+1+\epsilon} + 1)\, q^{k-t}\, \qnumb{r-k}{q}\text{.}
    \end{align*}
    This gives the stated expressions for $\lambda' = \#\Lambda$.
\end{proof}

\begin{remark}
    In the important case $k = r$, the formula in Theorem~\ref{thm:residual} simplifies to
    \[
        \lambda'
        = \begin{cases}
            (q^\epsilon + 1)\lambda & \text{if }r' = r\text{;}   \\
            0                       & \text{if }r' = r-1\text{.}
        \end{cases}
    \]
    The value $0$ for $r' = r-1$ has the following explanation.
    Here, $\PS'$ does not contain any subspaces of rank $r$ and thus, $\Res_H(D)$ will be the empty design.
    Therefore, for $k = r$ we would like to apply Theorem~\ref{thm:residual} for hyperplanes with $r' = r$ (which exist only for the types
    $U^-(q)$, $\Omega(q)$ and $\Omega^-(q)$).
    For this situation, the construction of Theorem~\ref{thm:residual} is already found in \cite[Thm.~3.3]{cossidente2021regular}.
\end{remark}

\begin{theorem}[{{\cite[Thm.~3.3]{cossidente2021regular}}}]\label{thm:superresidual}
	Assume $\mathcal{Q}$ is embedded as a non-degenerate hyperplane in a polar space $\bar{\mathcal{Q}}$ of rank $r + 1$.
	Denote the type of $\bar{\mathcal{Q}}$ by $\bar{Q}$.
	Let $D$ be a $t$-$(r,r,\lambda)_Q$ design in $\mathcal{Q}$ and
	\[
		\bar{D} = \{\bar{K}\in\qbinom{\bar{\mathcal{Q}}}{r+1}{\bar{Q}} \mid \bar{K} \text{ contains a block of } D\}\text{.}
	\]
	Then $\bar{D}$ is a $t$-$(r+1,r+1,(1+q^{\epsilon})\lambda)_{\bar{Q}}$ design.
\end{theorem}
In the theorem, the possible combinations for $(\mathcal{Q},\bar{\mathcal{Q}})$ are $(\Omega^-(2r+2,q),\Omega(2(r+1) + 1,q))$, $(\Omega(2r + 1,q),\Omega^+(2(r+1),q)$
and $(U(2r+1,q),U(2(r+1),q))$.
The construction is based on the fact that in these situations, every generator of $\bar{\mathcal{Q}}$ contains a unique generator of $\mathcal{Q}$.
We remark that our residual design in Theorem~\ref{thm:residual} is the derived design of $\bar{D}$ in Theorem~\ref{thm:superresidual}.

\begin{remark}
    While the polar space version of the derived design in Theorem~\ref{thm:derived} is a straightforward generalization of the subspace version in~\cite{Kiermaier-Laue-2015}, our approach to the residual design in polar spaces in Theorem~\ref{thm:residual} feels somewhat more distant.
    The reason is that it involves a change of the type of the ambient polar space.
    One might look for a similar construction where the ambient space is restricted to a proper subspace $U$ of the same type.
    The largest possible and hence natural choice of such a subspace is $U = L^\perp$ with a hyperbolic line $L$.
    Restricting the polar space $\PS$ to $U$, the rank drops by one.
    Therefore, for a meaningful definition, one would expect that also the strength of the resulting notion of a residual design drops by one compared to the starting design on $\PS$.%
    \footnote{The residual of a classical $t$-$(v,k,\lambda)$ design on a $v$-element set $V$ has the parameters $(t-1)$-$(v-1,k,\lambda')$ with a certain value $\lambda'$.
    In the polar space setting, the role of $v$ is taken over by the rank $r$.
    Hence, we would expect that the residual of a $t$-$(r,k,\lambda)_Q$ polar space design has the parameters $(t-1)$-$(r-1,k,\lambda'')_Q$.}

    However, this is not generally the case.
    We checked computationally that for the constructed $2$-$(4,3,3)_{\Omega^+(2)}$ design (see Section \ref{sec:results}), only $15$ out of the $4320$
    hyperbolic lines lead to a $1$-$(3,3,2)_{\Omega^+(2)}$ design, and for the constructed $2$-$(4,3,6)_{\Omega(2)}$ design, no hyperbolic line leads to a $1$-design.
\end{remark}

So far, we have analogs of the supplementary, the reduced, the derived and the residual design in polar spaces.
In the theory of block and subspace designs, there is a further general construction, the dual (or complementary) design.
However, we did not find a convincing counterpart of that construction.
Our feeling is that there might not be such a thing, since the classical construction of the dual design is based on the fact that the
underlying subset resp.~subspace lattice is self-dual, which however is not true for the poset of all subspaces of a polar space.

\pagebreak

\section{Intersection numbers}\label{sec:intersection}

Another classical topic in the theory of combinatorial and subspace designs are the intersection numbers, which describe the intersection sizes of the
blocks of a design with a fixed subspace $S$.
For combinatorial designs they have been originally defined in \cite{Mendelsohn-1971} for blocks $S$ and independently as ``$i$-Treffer'' for
general sets $S$ in \cite{Oberschelp-1972-MPSemBNF19:55-67}.

Lemma~\ref{lem:reduced} allows us to develop the corresponding result for designs in polar spaces.
In fact, the statements and the proofs can be done almost literally in the same way as for subspace designs in \cite{Kiermaier-Pavcevic-2015-JCD23[11]:463-480}.

For a fixed $t$-$(r,k,\lambda)_Q$ design $D$ in the polar space $\mathcal{Q}$ and a subspace $S$ of $\mathcal{Q}$, we define the
$i$-th \emph{intersection number} ($i\in\{0,\ldots,k\}$) as
\[
	\alpha_i(S) = \#\{B\in D \mid \dim(B \cap S) = i\}\text{.}
\]
If the space $S$ is clear from the context, we use the abbreviation $\alpha_i = \alpha_i(S)$.

First, we derive a polar space counterpart of the \emph{intersection equations} or \emph{Mendelsohn equations},
see \cite[Th.~1]{Mendelsohn-1971}, \cite[Satz~2]{Oberschelp-1972-MPSemBNF19:55-67} for classical designs and
\cite[Th.~2.4]{Kiermaier-Pavcevic-2015-JCD23[11]:463-480} for subspace designs.

\begin{theorem}[Intersection equations]
	\label{thm:mendelsohn}
	Let $D$ be a $t$-$(r,k,\lambda)_Q$ subspace design and $S$ a subspace of $\mathcal{Q}$ of dimension $s = \dim(S)$.
	For $i\in\{0,\ldots,t\}$ we have the following equation on the intersection numbers of $S$ with respect to $D$:
	\[
		\sum_{j = i}^s \qbinom{j}{i}{q} \alpha_j = \qbinom{s}{i}{q}\lambda_i\text{.}
	\]
\end{theorem}

\begin{proof}
We count the set $X$ of all pairs $(I,B)\in \qbinom{V}{i}{q} \times \mathcal{B}$ with $I \leq B\cap S$ in two ways:
On the one hand, there are $\qbinom{s}{i}{q}$ possibilities for the choice of $I\in\qbinom{S}{i}{q}$.
By Lemma~\ref{lem:reduced}, there are $\lambda_i$ blocks $B$ such that $I \leq B$, which shows that $\#X$ equals the right-hand side of the stated equation.
On the other hand, fixing a block $B$, the number of $i$-subspaces $I$ of $B\cap S$ is $\qbinom{\dim(B\cap S)}{i}{q}$.
Summing over the possibilities for $j = \dim(B\cap S)$, we see that $\#X$ also equals the left-hand side of the stated equation.
\end{proof}

The intersection equations can be read as a linear system of equations $Ax = b$ on the \emph{intersection vector} $x = (\alpha_0, \alpha_1,\ldots, \alpha_k)$.
The left $(t+1)\times(t+1)$ square part of the matrix $A = (\qbinom{j}{i}{q})_{ij}$ $(i\in\{0,\ldots,t\}$ and $j\in\{0,\ldots,k\})$ is called
the \emph{upper triangular $q$-Pascal matrix}, which is known to be invertible with inverse matrix $\left( (-1)^{j-i} q^{\binom{j-i}{2}} \qbinom{j}{i}{q}\right)_{ij}$.
Left-multiplication of the equation system with this inverse yields a parameterization of the intersection numbers $\alpha_0,\ldots,\alpha_t$ by $\alpha_{t+1},\ldots,\alpha_k$.
In this way, we get a counterpart of the Köhler equations, see \cite[Satz~1]{Koehler-1988_1989-DM73[1-2]:133-142} for classical
designs and \cite[Th. 2.6]{Kiermaier-Pavcevic-2015-JCD23[11]:463-480} for subspace designs.

\pagebreak

\begin{theorem}[Köhler equations]
\label{thm:koehler}
Let $D$ be a $t$-$(r,k,\lambda)_Q$ subspace design and $S$ a subspace of $V$ of dimension $s = \dim(S)$.
For $i\in\{0,\ldots,t\}$, a parametrization of the intersection number $\alpha_i$ by $\alpha_{t+1},\ldots,\alpha_k$ is given by
\[
	\alpha_i
	= \qbinom{s}{i}{q}\sum_{j = i}^t (-1)^{j-i} q^{\binom{j-i}{2}} \qbinom{s-i}{j-i}{q}\lambda_j
	+ (-1)^{t+1-i} q^{\binom{t+1-i}{2}} \sum_{j = t+1}^k \qbinom{j}{i}{q} \qbinom{j-i-1}{t-i}{q}\alpha_j\text{.}
\]
\end{theorem}

\begin{corollary}\label{cor:intersection_steiner}
	Let $D$ be a $t$-$(r,k,1)_Q$ Steiner system.
	Let $S\in D$ be a block.
	Then the intersection vector of $S$ is uniquely determined by $\alpha_k = 1$, $\alpha_{t+1} = \ldots = \alpha_{k-1} = 0$ and the
    values $\alpha_i$ ($i\in\{0,\ldots,t\}$) as in Theorem~\ref{thm:koehler}.
\end{corollary}

\begin{proof}
	Clearly, $\alpha_k = 1$.
	Assume that $\alpha_j \neq 0$ for a $j\in\{t+1,\ldots,k-1\}$.
	Then there exists a block $B\in D$, $B \neq S$ with $\dim(B \cap S) = j > t$.
	Now any $t$-subspace of $B \cap S$ is contained in the two blocks $B \neq S$ in contradiction to $\lambda = 1$.

	So $\alpha_{t+1} = \ldots = \alpha_{k-1} = 0$, and the $\lambda_i$ with $i\in\{0,\ldots,t\}$ are determined by Theorem~\ref{thm:koehler}.
\end{proof}

In hyperbolic polar spaces $\Omega^+(q)$, there is a unique partition of the generators into two parts, commonly
called \emph {Latins} and \emph{Greeks}, which are a design with the parameters $(r-1)$-$(r,r,1)_{\Omega^+(q)}$, see e.g. \cite{deBeuleKleinMetsch2011}.
We will call them legs of the \emph{Latin-Greek halvings}, or simply the Latin-Greek halvings.
(Note that by $\lambda_{\max} = q^{1+\epsilon} + 1 = 2$, they are halvings.)
For two generators $B$ and $B'$, the number $r - \dim(B \cap B')$ is even if and only if $B$ and $B'$ are contained in the same leg of the Latin-Greek halving.
Hence the intersection numbers of the Latin-Greek halving with respect to a block $S$ have the property $\alpha_i = 0$ if $r-i$ is odd.

The theory of intersection numbers yields the following uniqueness result.

\begin{lemma}\label{lem:latin_greek_unique}
	Let $D$ be a $(r-1)$-$(r,r,1)_{\Omega^+(q)}$ Steiner system.
	Then $D$ is a Latin-Greek halving.
\end{lemma}

\begin{proof}
	Let $S$ be a block of $D$.
	By Corollary~\ref{cor:intersection_steiner}, the intersection vector of $S$ is uniquely determined by the parameters of the Steiner system $D$,
    and hence it equals the corresponding intersection vector of the blocks of the Latin-Greek halving.
	Therefore $\alpha_i(S) = 0$ if $r - i$ is odd.
	This implies that $D$ contains only (and then all the) blocks from the same leg as $S$.
\end{proof}

\section{Fisher's inequality and symmetric designs}\label{sec:fisher}

For combinatorial designs, a classical result is Fisher's inequality~\cite{Bose1949,handbook-2006}, which states that a
$t$-$(v,k,\lambda)$ design of strength $t \geq 2$ has at least as many blocks as points.
In the case of equality, a design is called \emph{symmetric}.
For subspace designs, Fisher's inequality has been proven in~\cite{Cameron-1974}, also showing that there are no symmetric subspace designs of strength $t \geq 2$.
For generalizations of Fisher's inequality for classical and subspace designs, see \cite{Kiermaier-Wassermann-2023-GlasMatSerIII58[2]:181-200}.

\pagebreak

The natural question in our situation is if there exist designs such that
\begin{equation}\label{eq:fisher_ineq}
	\text{number of blocks} < \text{number of points of the ambient polar space}\text{.}
\end{equation}
Moreover, we would also like to investigate the equality case, where the corresponding designs are called \emph{symmetric}.
We start by looking at the consequences of Lemma~\ref{lem:reduced}.

\begin{lemma}\label{lem:fisher_admissibility}
	For $r \geq t \geq 2$ and $\lambda\in\{1,\ldots,\lambda_{\max}\}$, the admissible parameters $t$-$(r,k,\lambda)_Q$ such that the number of blocks
    is less or equal than the number of points of $\mathcal{Q}$, are given by
	\begin{itemize}
		\item $2$-$(r,r,1)_{\Omega(q)}$ and $2$-$(r,r,1)_{W(q)}$ with $r \geq 2$,
		\item $2$-$(r,r,\lambda)_{U^+(q)}$ with $\lambda = 1$ for $r = 2$, or $\lambda\in\{1,\ldots,\sqrt{q}\}$ for $r \geq 3$,
		\item $2$-$(r,r,\lambda)_{\Omega^+(q)}$ with $\lambda = 1$ for $r = 2$, or $\lambda \in \{1,2\}$ for $r = 3$, or $\lambda\in\{1,\ldots,q+1\}$ for $r \geq 4$,
		\item $3$-$(r,r,1)_{\Omega^+(q)}$ with $r\in\{3,4\}$,
	\end{itemize}
	where equality occurs precisely for
	\begin{itemize}
		\item $2$-$(2,2,1)_{\Omega(q)}$ and $2$-$(2,2,1)_{W(q)}$,
		\item $3$-$(4,4,1)_{\Omega^+(q)}$ and its reduced parameters $2$-$(4,4,q+1)_{\Omega^+(q)}$.
	\end{itemize}
\end{lemma}

\begin{proof}
The condition \enquote{number of blocks less or equal than number of points} means $\lambda_0 \leq \qbinom{r}{1}{Q}$, which (using $t \geq 1$) is equivalent to
\[
	\lambda \leq \frac{q^k - 1}{q-1}\cdot\frac{(q^{k-1} - 1) \cdots (q^{k-t+1} - 1)}{(q^{r-1} - 1) \cdots (q^{r-t+1} - 1)}\cdot \frac{1}{(q^{r+\epsilon-t+1} + 1) \cdots (q^{r+\epsilon-1} + 1)}
\]
The right-hand side is strictly decreasing in $r$, $\epsilon$, and $t$ and strictly increasing in $k$.
In the case $k = r-1$ we get (using $t \geq 2$)
\[
	\lambda \leq \frac{1}{q-1}\cdot \frac{q^{r-t}-1}{q^{r+\epsilon-1} + 1} \cdot \frac{1}{(q^{r+\epsilon-t+1} + 1) \cdots (q^{r+\epsilon-2} + 1)} < 1\text{,}
\]
such that we can restrict ourselves to the case $k = r$, where the investigated inequality simplifies to
\[
	\lambda \leq \frac{q^r - 1}{q-1} \cdot \frac{1}{(q^{r+\epsilon-t+1} + 1) \cdots (q^{r+\epsilon-1} + 1)}\text{.}
\]
We consider the following cases:
\begin{itemize}
	\item For $t \geq 4$ we have $r \geq 4$ and thus
	\begin{multline*}
		(q^{r+\epsilon-3} + 1)(q^{r+\epsilon-2} + 1)(q^{r+\epsilon-1} + 1) \\
		\geq (q^{0} + 1)\cdot q^{1} \cdot q^{r-2}
		= 2q^{r-1} > \frac{q^r - 1}{q-1}\text{,}
	\end{multline*}
	which implies $\lambda < 1$.
	\item
	In the case $t = 3$ and $\epsilon = -\frac{1}{2}$ (so $r \geq 3$ and $q \geq 4$) we have
	\begin{multline*}
		(q^{r+\epsilon-2} + 1)(q^{r+\epsilon-1} + 1)
		\geq (q^{1/2} + 1)\cdot q^{r - 3/2} \\
		= q^{r - 1} + \sqrt{q}q^{r-2}
		\geq q^{r-1} + 2q^{r-2} > \frac{q^r - 1}{q-1}\text{.}
	\end{multline*}
	\item
	In the case $t = 3$ and $\epsilon \geq 0$ (so $r \geq 3$) we have
	\[
		(q^{r+\epsilon-2} + 1)(q^{r+\epsilon-1} + 1)
		\geq q\cdot q^{r - 1} \\
		= q^r
		> \frac{q^r - 1}{q-1}\text{.}
	\]
	\item
	For $t = 3$ and $\epsilon = -1$ the inequality simplifies to
	\[
		\lambda \leq \frac{q^r - 1}{(q-1)(q^{r-3} + 1)(q^{r-2} + 1)}\text{.}
	\]
	For $r \geq 5$, $(q^{r-3} + 1)(q^{r-2} + 1) > q^2 \cdot q^{r-2} = q^r > \frac{q^r - 1}{q-1}$ and thus $\lambda < 1$.

	For $r = 4$, we get $\lambda \leq \frac{q^4 - 1}{(q-1)(q+1)(q^2+1)} = 1$, and $\lambda = 1$ is an equality case.

	For $r = 3$, we compute $\lambda_{\max} = 1$ and verify that the inequality holds for $\lambda = 1$.
\end{itemize}

	There remain the cases with $t = 2$, where the investigated inequality simplifies further to
	\[
		\lambda \leq \frac{q^r - 1}{(q-1)(q^{r+\epsilon - 1} + 1)} \eqqcolon \bar{\lambda}\text{.}
	\]
	The actual upper bound for $\lambda$ will be $\min(\lfloor \bar{\lambda}\rfloor, \lambda_{\max})$, and there is an equality case
    if and only if $\lambda = \bar{\lambda} \in\mathbb{Z}$.
\begin{itemize}
	\item
	For $\epsilon \geq \frac{1}{2}$ we have
	\[
		(q-1)(q^{r+\epsilon-1} + 1) > (q-1)q^{r+\epsilon-1} \geq q^r > q^r - 1
	\]
	where the \enquote{$\geq$} step follows from $q-1 \geq 1$ for $\epsilon = 1$ and from $q-1 \geq \sqrt{q}$ for $\epsilon = \frac{1}{2}$ (where we have $q \geq 4$).
	So $\bar{\lambda} < 1$.
	\item For $\epsilon = 0$, we get
	\[
		\bar{\lambda} = \frac{q^r - 1}{(q-1)(q^{r-1} + 1)} \leq \frac{q^{r-1} + q^{r-2} + \ldots + 1}{q^{r-1} + 1} = 1 + \frac{q^{r-2} + \ldots + q}{q^{r-1} + 1}\text{.}
	\]
	For $r = 2$, $\bar{\lambda} = 1$ (and we get the equality case $\lambda = 1$), and for $r \geq 3$, $1 < \bar{\lambda} < 2$.
	\item
	For $\epsilon = -\frac{1}{2}$ and $r = 2$ we have $\lambda_{\max} = 1$ and $\bar{\lambda} = \frac{q^2 - 1}{(q-1)(\sqrt{q} + 1)} = \frac{q+1}{\sqrt{q} + 1} > 1$.
	\item
	For $\epsilon = -\frac{1}{2}$ and $r \geq 3$,
	\begin{multline*}
		\bar{\lambda}
		= \frac{q^r - 1}{(q-1)(q^{r-3/2} + 1)} \\
		= \frac{\sqrt{q}^{2r} + \sqrt{q}^{2r-2} + \ldots + 1}{\sqrt{q}^{2r-3} + 1}
		= \sqrt{q} + \frac{(\sqrt{q}^{2r - 4} + \ldots + 1) - \sqrt{q}}{\sqrt{q}^{2r - 3} + 1}
	\end{multline*}
	is strictly between $\sqrt{q}$ and $\sqrt{q} + 1$.
	Hence $\lambda \leq \sqrt{q}$.
	Furthermore, note that $\lambda_{\max} = (q^{1/2} + 1) \cdots (q^{r-5/2} + 1) > \sqrt{q} = \bar{\lambda}$.
	\item
	For $\epsilon = -1$ and $r = 2$ we have $\bar{\lambda} = \frac{q^r-1}{q-1} > 1$ and $\lambda_{\max} = 1$.
	\item For $\epsilon = -1$ and $r \geq 3$,
	\[
		\bar{\lambda}
		= \frac{q^{r-1} + q^{r-2} + \ldots + 1}{q^{r-2} + 1}
		= q + \frac{(q^{r-2} + \ldots + q^2) + 1}{q^{r-2} + 1}\text{.}
	\]
	For $r = 3$, this shows $q < \bar{\lambda} < q+1$.
	For $r \geq 4$, we compute further
	\[
		\bar{\lambda} = q + 1 + \frac{q^{r-3} + \ldots + q^2}{q^{r-2} + 1}\text{.}
	\]
	For $r = 4$, we get the equality case $\bar{\lambda} = q+1$, and for $r \geq 5$ we have $q+1 < \bar{\lambda} < q+2$.

	Furthermore, we note that $\lambda_{\max} = 2$ for $r = 3$ and $\lambda_{\max} = 2(q + 1) \cdots (q^{r-2} + 1) > q+1$ for $r \geq 4$.
\end{itemize}
\end{proof}

\begin{remark}
	The equality cases in Lemma~\ref{lem:fisher_admissibility} show that symmetric designs in polar spaces can only exist in a few sporadic cases.
	\begin{itemize}
		\item The respective complete designs are the unique designs with parameters $2$-$(2,2,1)_{\Omega(q)}$ and $2$-$(2,2,1)_{W(q)}$.
		\item The respective Latin-Greek design is the unique symmetric design with the parameters $3$-$(4,4,1)_{\Omega^+(q)}$, see Lemma~\ref{lem:latin_greek_unique}.
		\item Its reduced design is symmetric with the parameters $2$-$(4,4,q+1)_{\Omega^+(q)}$.
		We would like to point out that this design is not necessarily unique, as we found symmetric $2$-$(4,4,3)_{\Omega^+(2)}$ designs whose
        block intersection vectors differ from the one of the Latin-Greek halving.
	\end{itemize}

	Furthermore, Lemma~\ref{lem:fisher_admissibility} lists candidates for the parameters of designs in polar spaces for which inequality~\eqref{eq:fisher_ineq} holds.
	We see that such designs can only exist for types $\Omega^+(q)$ and $U^+(q)$.
	Known examples are the complete $2$-$(2,2,1)_{\Omega^+(q)}$ and the Latin-Greek $2$-$(3,3,1)_{\Omega^+(q)}$ designs.
\end{remark}

The common textbook proof of Fisher's inequality for classical and subspace designs uses Bose's elegant argument based on the rank
of the point-block incidence matrix $N$ \cite{Bose1949}.
We follow this approach for our case of designs in polar spaces.

For a combinatorial $2$-$(v,k,\lambda)$ design (or a subspace design) it is easy to see that
\[
    NN^\top = (\lambda_1-\lambda)I + \lambda J\,,
\]
where $J$ is the all-one matrix.
Now, let $D$ be a $2$-$(r,k,\lambda)_Q$ design and
let $N$ be the point-block incidence matrix of $D$.
Then for two points $x,y\in\qbinom{\PS}{1}{Q}$ we have
\[
    (NN^\top)_{xy}
    = \begin{cases}
        \lambda_1 & \text{for } x = y\text{;}                   \\
        \lambda   & \text{if } x\text{ and }y\text{ collinear;} \\
        0         & \text{otherwise.}
    \end{cases}
\]
It follows that
\[
    NN^\top = \lambda_1 I + \lambda\cdot A\,,
\]
where $A$ is the adjacency matrix of the collinearity graph of the polar space.

The equation for $NN^\top=\lambda_1 I + \lambda\cdot A$ for designs in polar spaces
is a straightforward generalization of the equation
$NN^\top = (\lambda_1 -\lambda)I + \lambda J$ for subspace designs, since
$J-I$ can be regarded as the adjacency matrix of the collinearity graph, i.e.~the complete graph, in PG$(V,q)$.

\begin{lemma}\label{lem:rk_N}
Let $D$ be a $t$-$(r,k,\lambda)_Q$ design of strength $t \geq 2$.
Then the rank of the point-block incidence matrix $N$ of $D$ is
\[
	\rank(N) = \begin{cases}
		1 + m_1 & \text{for }k=r\text{;}\\
		\qbinom{r}{1}{Q} & \text{for }k < r\text{.}
	\end{cases}
\]
\end{lemma}

\begin{proof}
The eigenvalues $\mu_i$ of $NN^\top$ can be derived from the eigenvalues $\theta_i$ of $A$ as
\[
    \mu_i = \lambda_1+\lambda\theta_i
\]
with multiplicities $m_i$, $i=0,1,2$, see also \cite{Bailey2022}.
Since $\lambda, \lambda_1>0$ also $\mu_0, \mu_1>0$
and a simple calculation shows that
$\mu_2=0$ if and only if $k=r$, independent from $\lambda$.
In this case, the rank of the matrices $NN^\top$ and hence $N$ over the real numbers is equal to $1+m_1$.
In all other cases, the matrix $NN^\top$ and hence $N$ has full rank.
\end{proof}

\begin{theorem}[Fisher's inequality for designs in polar spaces]\label{thm:fisher}~\linebreak
    Let $D$ be a $t$-$(r,k,\lambda)_{Q}$ design of strength $t\geq2$. If $k<r$, then $\lambda_0 \geq \qbinom{r}{1}{Q}$, i.e.~the number of
    blocks is greater or equal than the number of points in the polar space. If $k=r$, then
    $\lambda_0 \geq 1 + m_1$.
\end{theorem}

\begin{proof}
    Let $N$ be the point-block incidence matrix of $D$.
    Since $N$ is a $\qbinom{r}{1}{Q} \times \lambda_0$ matrix, $\rank N \leq \lambda_0$.
    If $k < r$, by Lemma~\ref{lem:rk_N} the matrix $N$ has full rank, i.e.~$\rank N = \qbinom{r}{1}{Q}$.
    If $r=k$, $\rank N = 1+m_1$.
\end{proof}

For $k < r$, Fisher's inequality follows already from Lemma~\ref{lem:fisher_admissibility}.
However, for $k = r$, Fisher's inequality~\ref{thm:fisher} sharpens Lemma~\ref{lem:fisher_admissibility}, which leads to the following classification result.

\begin{theorem}\label{thm:steiner_omegaplus}
	Let $D$ be a $t$-$(r,r,1)_{\Omega^+(q)}$ Steiner system of strength $t \geq 2$.
	Then $D$ is either a complete or a Latin-Greek design.
\end{theorem}

\begin{proof}
	If $t = r$, $D$ is a complete design, and if $t = r-1$, $D$ is a Latin-Greek design by Lemma~\ref{lem:latin_greek_unique}.

	For $t \leq r - 2$, by Theorem~\ref{thm:derived} the $(t-2)$-fold derived design of $D$ is a Steiner system with the parameters
    $2$-$(r',r',1)_{\Omega^+(q)}$ where $r' = r-t+2 \geq 4$, such that we may assume $t = 2$ and $r \geq 4$.
	Comparing
	\[
		m_1 = \frac{q(q^{r-2}+1)(q^r - 1)}{q^2 - 1}
		\quad\text{and}\quad
		\lambda_0 = (q^{r-2} + 1)(q^{r-1} + 1)\text{.}
	\]
	we get
	\[
		\frac{m_1}{\lambda_0}
		= \frac{q(q^r - 1)}{q^{r-1} + 1)(q^2 - 1)}
		= \frac{q^{r+1} - q}{q^{r+1} - q^{r-1} + q^2 - 1}
		> 1\text{,}
	\]
	where in the last step we used $r \geq 4$.
	Hence $m_1 > \lambda_0$, which is a contradiction to Fisher's inequality~\ref{thm:fisher}.
\end{proof}

An important part of our proof of Theorem~\ref{thm:steiner_omegaplus} is to show the non-existence of $t$-$(r,r,1)_{\Omega^+(q)}$ designs with $t \geq 2$ and $r \geq 4$.
This result follows also from non-existence stated in Theorem~\ref{thm:schmidtweiss} below, which has been proven by Schmidt and Wei\ss{} using
the machinery of association schemes.
Our argument based on Fisher's inequality provides an elementary proof of that special case.

\begin{theorem}[Schmidt, Wei\ss{} \cite{Schmidt-Weiss-2022}]\label{thm:schmidtweiss}
    Suppose there exists a
    Steiner system $t$-$(r, r, 1)_Q$ with $t\in\{2,\ldots,r-1\}$.
    Then one of the following holds

    \begin{enumerate}[(a)]
        \item $t=2$ and $r$ odd and $Q$ is either $U(q)$ or $\Omega^{-}(q)$,
	\item $t = r-1$ and $q\neq 2$ and $Q$ is either $U(q)$ or $\Omega^{-}(q)$,
        \item $t=r-1$ and $Q = \Omega^+(q)$.
    \end{enumerate}
\end{theorem}

For other non-existence results see \cite{cossidente2021regular}.
As $t$-$(r,r,\lambda)_Q$ designs are $t$-designs in association schemes (compare~\cite{vanhove2011}, \cite[Prop~2.2]{cossidente2021regular}), Delsarte's general theory can be applied and gives the following stronger version of Fisher's inequality.

\begin{theorem}[{{Delsarte \cite[Thm 5.19]{Delsarte-1973}}}]
    For every $t$-$(r,r,\lambda)_Q$ design $D$ we have
\[
    \# D \geq \sum_{i=0}^{\lfloor t/2\rfloor} \mu_i,
\]
where $\mu_i$ is the $i$-th multiplicity of the corresponding association scheme.
\end{theorem}

For polar spaces, the multiplicities $\mu_i$ have been determined by
Stanton~\cite[Cor. 5.5]{Stanton-1980}.
For $t=2,3$, Delsarte's theorem gives $\# D \geq 1 + \mu_1$ and $\mu_1$ equals the multiplicity $m_1$ of one of the eigenvalues
of the collinearity graph in Theorem \ref{thm:fisher}.

K.-U.~Schmidt was confident that Delsarte's inequality for designs in polar spaces could be extended to
$k < r$, too. Moreover, Delsarte's inequality is a special case of the linear programming bound
for designs in association schemes, and he conjectured that the linear programming bound
allows to extend the non-existence
results in Theorem~\ref{thm:schmidtweiss} from $\lambda=1$ to something like
$\lambda\leq q^{\binom{t}{2}}$.
This conjecture is supported by
our computational results in Section \ref{sec:results}. Unfortunately, after his tragic, untimely
death his latest results in this regard seem to be lost.

\section{Computational construction of designs}\label{sec:compute}

For the construction of designs we use the well-known approach
to choose a group $G$ and search for $G$-invariant designs.
Thus, instead of searching for individual $k$-subspaces
the design has to be composed of $G$-orbits on the $k$-subspaces.
This approach applied to combinatorial designs is commonly attributed to
Kramer and Mesner \cite{Kramer-Mesner-1976-DM15[3]:263-296} and has been used there with great success.

For a subgroup $G$ of the collineation group of a polar space $\PS$, the set of blocks of a
$G$-invariant $t$-$(r,k,\lambda)_Q$ design ($G\le \Aut(\PS)$) is the disjoint union of
orbits of $G$ on the set $\qbinom{\PS}{k}{Q}$ of $k$-dimensional subspaces of $\PS$.
To obtain an appropriate selection of orbits of $G$ on $\qbinom{\PS}{k}{Q}$ we consider the incidence matrix $A_{t,k}^G$
whose rows are indexed by the $G$-orbits on the set of $t$-subspaces of $\PS$ and whose columns are indexed by the
orbits on $k$-subspaces. The entry  $a_{T,K}^G$ of $A_{t,k}^G$ corresponding to the orbits $T^G$ and $K^G$ is defined by
\[
    a_{T,K}^G:=\#\{K'\in K^G\mid T\le K'\}.
\]
Now the theorem by Kramer and Mesner, originally formulated for combinatorial designs, can be applied to
designs in polar spaces as well:
\begin{theorem}[{Kramer, Mesner \cite{Kramer-Mesner-1976-DM15[3]:263-296}}]\label{thm:KramerMesner}
    Let $\PS$ be a polar space and $G$ be a subgroup of its collineation group.
    There exists a $G$-invariant $t\text{-}(r,k,\lambda)_Q$ design
    if and only if there is a $0/1$-vector $x$ satisfying
    \begin{equation}
        A_{t,k}^G \cdot x =
        \left(\begin{array}{c}
                \lambda \\
                \vdots  \\
                \lambda
            \end{array}\right)\text{.}
        \label{eq:KM}
    \end{equation}
\end{theorem}

The question now is which are the promising groups that can be prescribed in Theorem \ref{thm:KramerMesner}?
For combinatorial designs, cyclic groups have turned out to be promising candidates for prescribed groups in
Theorem~\ref{thm:KramerMesner}. For subspace designs, the Singer cycle and its normalizer, and subgroups thereof are good
candidates \cite{qdesignscomputer2017}. For designs in polar spaces, it is not yet clear if there are
exceptionally promising groups to be used for Theorem~\ref{thm:KramerMesner}.

With these previous experiences in mind, we constructed all cyclic subgroups of the automorphism group
of each polar space together with its normalizer using MAGMA \cite{magma}. The cyclic groups and random subgroups of their normalizers
were prescribed as automorphism groups in Theorem~\ref{thm:KramerMesner}.

The systems of Diophantine linear equations (\ref{eq:KM}) were attempted to be solved by the third author's software
\texttt{solvediophant} \cite{wassermann98,Wassermann-2021}. The parameters for which designs have been found
are listed in Section~\ref{sec:results}.
In small cases, exhaustive non-existence results could be achieved.
So far, no designs could be found in the Hermitian polar spaces. However, this seems rather to be due to
the huge size of the search space than due to the scarcity of designs.

\subsection{Results}\label{sec:results}
Table \ref{tab:designs} lists the parameters of the constructed designs.
An asterisk indicates that a design with these parameters can be constructed from a Latin-Greek halving,
after possibly applying one or several repeated transitions to the derived and/or reduced design.
An italic number indicates that the parameters of that design have not been found by the computer search
but can be constructed via
Theorem~\ref{thm:superresidual} from another known design (including the newly constructed designs).
For all listed parameters the realizability result is new, except the ones marked by an asterisk,
the designs in $\Omega^-(8, 2)$ and $\Omega(7, 3)$ found by De Bruyn and Vanhove \cite{DeBruyn2021},
and their descendants in $\Omega(9,2)$, $\Omega^+(10,2)$ and $\Omega^+(8,3)$ in terms of
Theorem~\ref{thm:superresidual}.
We have published the constructed designs in the open-source research data repository \cite{Kiermaier-Wassermann-2024-Zenodo:13353388}.

\begin{table}[!htbp]
\caption{List of known realizable $2$-$(r,k,\lambda)_Q$ designs for $q=2,3$}\label{tab:designs}
\begin{center}
    \small
    \begin{tabular}{llllrlp{7cm}}
        \hline
        \multicolumn{7}{c}{$\mathbf \QQm$}                                                                             \\\hline
        $q$ & $r$ & $k$         & $\Deltalambda$ & $\lambda_{\max}$ & $\nexists \lambda$ & $\exists\lambda$                  \\ \hline
        $2$ &
        $3$ & $3$         & $1$            & $5$              & 1                  & 2  \quad(De Bruyn, Vanhove)            \\
        & $4$ & $\mathbf 3$ & $3$            & $27$             &                    & 6, 9, 12                          \\
        & $4$ & $4$         & $1$            & $45$             & 1                  & 9, 11, 12, 14, 15, 16, 18, 19, 21 \\
        & $5$ & $5$         & $1$            & $765$            & 1                  & 240, 245, 275, 280, 315, 360      \\
        \hline
        $3$ &
        $3$ & $3$         & $1$            & $10$             & 1                  & 2, 5                              \\[2ex]
        \hline
        \multicolumn{7}{c}{$\mathbf \QQ$}                                                                              \\\hline
        $q$ & $r$ & $k$         & $\Deltalambda$ & $\lambda_{\max}$ & $\nexists \lambda$ & $\exists\lambda$                  \\ \hline
        $2$ &
        $3$ & $3$         & $1$            & $3$              & $1$                & -                                 \\
        & $4$ & $\mathbf 3$ & $1$            & $15$             &                    & 6, 7                              \\
        & $4$ & $4$         & $1$            & $15$             & $1$                & 5, 6, 7                           \\
        & $5$ & $5$         & $1$            & $135$            & $1$                & 21, 24, 27,
        29, 30, 32, 33, 35, 36, 39, 40, 42, 45, 47, 48,
        50, 51, 52, 53, 54, 55, 56, 57, 58, 60, 61, 62, 63, 64, 65, 66                                                 \\
        & $6$ & $6$         & $1$            & $2295$           & $1$                & \emph{720},
        \emph{735}, \emph{825}, \emph{840}, \emph{945}, \emph{1080}                                                    \\
        \hline
        $3$ &
        $3$ & $3$         & $1$            & $4$              & 1                  & 2 \quad(De Bruyn, Vanhove)             \\
        & $4$ & $4$         & $1$            & $40$             & 1                  & 8, 20                             \\[2ex]
        \hline
        \multicolumn{7}{c}{$\mathbf \QQp$}                                                                             \\\hline
        $q$ & $r$ & $k$         & $\Deltalambda$ & $\lambda_{\max}$ & $\nexists \lambda$ & $\exists\lambda$                  \\ \hline
        $2$ &
        $3$ & $3$         & $1$            & $2$              & -                  & $1^*$                             \\
        & $4$ & $\mathbf 3$ & $3$            & $9$              &                    & $3$                               \\
        & $4$ & $4$         & $1$            & $6$              & 1,2                & $3^*$                             \\
        & $5$ & $5$         & $1$            & $30$             & 1                  & 6, 8, 10, 12, 14, $15^*$          \\
        & $6$ & $6$         & $1$            & $270$            & 1                  & 32, 40,
        \emph{42}, 45,
        48, 50, 51, 52, 53, 54, 56, 57,
        58, 60, 62, 63, 64, 65, 66, 67, 69, 70, 72, 74, 75, 77, 78, 79,
        80, 81, 84, 85, 86, 87, 88, 90, 91, 92, 93, 94, 95, 96, 98, 99,
        100, 102, 103, 104, 105, \emph{106}, 107, 108, 109, 110, 111, 112, 114, 115, 116, 117, 118,
        119, 120, 121, 122, 123, 124, 125, 126, 127, 128, 129, 130, 132, 133, 134, $135^*$                             \\
        & $7$ & $7$         & $1$            & $4590$           & $1$                & \emph{1440},
        \emph{1470}, \emph{1650}, \emph{1680}, \emph{1890}, \emph{2160}                                                \\
        \hline
        $3$ &
        $3$ & $3$         & $1$            & $2$              & -                  & $1^*$                             \\
        & $4$ & $4$         & $1$            & $8$              & 1                  & $4^*$                             \\
        & $5$ & $5$         & $1$            & $80$             & 1                  & 8, 16, 32, $40^*$                 \\[2ex]
        \hline
        \multicolumn{7}{c}{\textbf{$\QWb$}}                                                                            \\ \hline
        $q$ & $r$ & $k$         & $\Deltalambda$ & $\lambda_{\max}$ & $\nexists \lambda$ & $\exists\lambda$                  \\ \hline
        $2$ &     &             &                &                  &                    & see $\Omega(2r+1,2)$ \\ \hline
        $3$ &
        $3$ & $3$         & $1$            & $4$              & 1, 2               & -  \quad(2 by De Bruyn, Vanhove)       \\
        & $4$ & $4$         & $1$            & $40$             & 1                  & 20                                \\
    \end{tabular}
\end{center}
\end{table}

\section{Open questions and future work}\label{sec:outlook}

We would like to conclude this article by pointing out some open questions and suggestions for future work.

In the theory of combinatorial and subspace designs, the numbers $\lambda_{i,j}$ of a $t$-design $D$ play an important role.
They can be defined in two variants (which coincide in the case of combinatorial designs).
Let $i,j$ be non-negative integers with $i+j \leq t$.
For subspaces $I,J$ with $\dim(I) = i$ and $\dim(J) = j$, the number of blocks $B \in D$ containing $I$ and having trivial intersection with $J$ only depends on $i$ and $j$ and yields the first variant of $\lambda_{i,j}$.
Similarly, for subspaces $I$, $J$ with $\dim(I) = i$ and $\codim(J) = j$, the number of blocks $B \in D$ with $I \subseteq B \subseteq J$ only depends on $i$ and $j$, which gives the second variant.
It might be interesting to investigate this concept for designs in polar spaces.

For combinatorial and subspace designs, a \emph{large set} is a partition of the set of all $k$-subspaces of the ambient space into $N$ $t$-designs of the same size.
Large sets generalize the notion of a halving (and thus of hemisystems in polar spaces), which are given by the case $N = 2$.
An important application of large sets of combinatorial and subspace designs is the construction of infinite series of designs, see \cite{Khosrovshahi-AjoodaniNamini-1991-JCTSA58[1]:26-34,Braun-Kiermaier-Kohnert-Laue-2017-JCTSA147:155-185}.
Maybe this strategy can also be applied to large sets of designs in polar spaces.

One could further investigate symmetric $2$-$(4,4,q+1)_{\Omega^+(q)}$ designs, which are the only parameters where the symmetric designs in polar spaces have not been fully classified in Section~\ref{sec:fisher}.
Moreover, we would like to ask if inequality~\eqref{eq:fisher_ineq} allows any designs in $U^+(q)$, and if it allows any non-complete and non-Latin-Greek designs in $\Omega^+(q)$.

For $q$ odd, it is an open question whether a spread in $\QQm$, $r\ge 3$ exists \cite[Sec.~7.5]{Hirschfeld-Thas}.
In our notation, these are the designs with the parameters $1$-$(r, k, 1)_{\Omega^-(q)}$.
By computer search, we could show that if in $\Omega^-(8, 3)$ such a spread exists,
the only prime powers that may divide the order of an automorphism are $2$ or $3$.

An $m$-dimensional Kerdock set over $\mathbb{F}_q$ with even $m$ is a set of $q^{m-1}$ skew-symmetric
$(m\times m)$-matrices such that the difference of any two of them is non-singular.
Its existence is equivalent to the existence of a $1$-$(m,m,1)_{\Omega^+(q)}$ spread.
It has been shown that for even $q$, the Kerdock sets do exist for all $m$.
For odd $q$, the existence for $m \geq 6$ is open.
We could show by a computer search that if a $1$-$(6,6,1)_{\Omega^+(3)}$ exists, its automorphism group is of order $1,2,3,4,5,6$ or $8$.

\section*{Acknowledgement}
The authors want to thank the anonymous referees for helpful comments.

Further, the authors gratefully acknowledge the Leibniz Supercomputing Centre for funding this project by providing
computing time on its Linux-Cluster.

The collaboration on this topic started with a discussion at the Oberwolfach Workshop 1912: Contemporary Coding Theory in March 2019.
The work on this article was almost finished, when our coauthor and friend, Kai, tragically passed away in August 2023.
We greatly miss him, not least for his sharp mind and charming personality.

\bibliographystyle{plain}
\bibliography{polar}

\end{document}